\documentclass{article}

\usepackage{mathtools}
\usepackage{structure}

\newcommand{\modified}[1]{#1}

\title{On the Gap between Hereditary Discrepancy and the Determinant Lower Bound}

\author{Lily Li\thanks{University of Toronto, Department of Computer Science
  (\href{mailto:xinyuan@cs.toronto.edu}{xinyuan@cs.toronto.edu}).}
\and Aleksandar Nikolov\thanks{University of Toronto, Department of Computer Science (\href{mailto:anikolov@cs.toronto.edu}{anikolov@cs.toronto.edu}).}}

\begin{document}

\maketitle

\begin{abstract}
The determinant lower bound of Lov\'{a}sz, Spencer, and Vesztergombi
[European Journal of Combinatorics, 1986] is a general way to
prove lower bounds on the hereditary discrepancy of a set system. In
their paper, Lov\'{a}sz, Spencer, and Vesztergombi asked if hereditary
discrepancy can also be bounded from above by a function of the determinant lower bound. This was answered in the negative by Hoffman,
and the largest known multiplicative gap between the two quantities
for a set system of $m$ subsets of a universe of size $n$ is on the order
of $\max\{\log n,\sqrt{\log m}\}$. On the other hand, building upon work of Matou\v{s}ek [Proceedings of the AMS, 2013], Jiang and Reis [SOSA, 2022] showed that this gap is always bounded up to constants by $\sqrt{\log(m)\log(n)}$. This is tight when $m$ is polynomial in $n$, but leaves open the case of large $m$. We show that the bound of Jiang and Reis is tight for nearly the entire range of $m$. Our proof amplifies the discrepancy lower bounds of a set system derived from the discrete Haar basis via Kronecker products.
\end{abstract}

\section{Introduction}
Let $X$ be a finite universe of elements, $\mathcal{S}$ a collection of subsets of $X$ called a \emph{set system}, and $\chi: X \rightarrow \{\pm 1\}$ a $\pm 1$ coloring of the elements of $X$. The discrepancy of $\mathcal{S}$ with respect to $\chi$, denoted $\disc(\mathcal{S}, \chi)$, is defined as $\max_{S \in \mathcal{S}}||\chi^{-1}(1)\cap S| - |\chi^{-1}(-1)\cap S||$ i.e.\ the largest difference between the number of elements colored differently in any set $S \in \mathcal{S}$. The \emph{discrepancy of $\mathcal{S}$}, denoted $\disc(\mathcal{S})$, is then $\min_{\chi: X \rightarrow \{\pm 1\}}\disc(\mathcal{S}, \chi)$ i.e.\ among all $\pm 1$ colorings of the elements of $X$, the least unequal we can make the most unequal set of $\mathcal{S}$. If $|X| = n$ and $|\mathcal{S}| = m$, let $X = [n]$ and $\mathcal{S} = \{S_1, ..., S_m\}$. For each set system $\mathcal{S}$, we can construct an associated $m \times n$ incidence matrix $\mm{A}_{\mathcal{S}}$: the entry in row $i$ and column $j$ of $\mm{A}_{\mathcal{S}}$ is equal to one if element $j$ is in $S_i$ and zero otherwise. We define the \emph{discrepancy of a real-valued matrix \(\mm{A}\)} as 
\[\disc(\mm{A}) \coloneqq \min_{\mm{x} \in \{\pm 1\}^{n}} \norm{\mm{A}\mm{x}}_{\infty},\] 
where $\norm{\cdot}_{\infty}$ denotes the $L_\infty$-norm of a vector. Using this definition, we see that $\disc(\mm{A}_{\mathcal{S}}) = \disc(\mathcal{S})$. Throughout this exposition, we will use $\mathcal{S}$ and its indicator matrix $\mm{A}_{\mathcal{S}}$ interchangeably.

We would like discrepancy to be a robust quantity, but it can be sensitive to slight modifications to the incidence matrix e.g. the discrepancy of the matrix $[\mm{A}, \mm{A}]$ is always zero regardless of $\disc(\mm{A})$. \modified{This motivates the definition of the hereditary discrepancy of a matrix $\mm{A}$ as
\[\herdisc(\mm{A}) \coloneqq \max_{\mm{B} \in \mathcal{S}}\disc(\mm{B}),\] 
where $\mathcal{S}$ is the set of all sub-matrices of $\mm{A}$. Notice that adding rows to a matrix can never decrease its discrepancy so it suffices for $\mathcal{S}$ to consist of sub-matrices whose columns are a subset of the {columns} of $\mm{A}$.} In a sense, this definition generalizes total unimodularity. It is easy to show that totally unimodular matrices (TUM)\footnote{A matrix $\mm{A}$ is TUM if every square submatrix of $\mm{A}$ has determinant in $\{-1, 0, 1\}$. A linear systems of the form $\mm{A}\mm{x} \geq \mm{b}$ for TUM $\mm{A}$, integral $\mm{b}$, and $0 \le \mm{x}$ has an integral polyhedron as its feasible region.} have hereditary discrepancy at most one~\cite{schrijver1998theory}. Further, a result of Ghouila-Houri~\cite{ghouila1962caracterisation} states that the set of hereditary discrepancy one matrices with entries in $\{-1, 0, 1\}$ is exactly the set of TUM matrices. 

An important early work in discrepancy theory of Lov\'{a}sz, Spencer, and Vesztergombi~\cite{Lovasz1986discrepancy} showed that the determinant lower bound of $\mm{A}$, \modified{defined as
\[\detlb(\mm{A}) \coloneqq \max_{k}\max_{\mm{B}\in \mathcal{S}_k}\left|\mathrm{det}(\mm{B})\right|^{1/k},\]
where $\mathcal{S}_{k}$ denotes the set of all $k\times k$ sub-matrices of $\mm{A}$}, satisfies $2\herdisc(\mm{A}) \geq \detlb(\mm{A})$. This, again, generalizes what happens with totally unimodular matrices, for which both quantities are equal to one. The determinant lower bound has since become a powerful tool in proving nearly tight lower bounds on many natural and important set systems, e.g. axis-aligned boxes and point-line incidences \cite{chazellelvov2001discrepancy, matousek18factorization}. Given that the determinant lower bound often implies nearly tight discrepancy lower bounds, it is natural to ask how far it can be from the hereditary discrepancy. The first result in this direction is due to Matou\v{s}ek~\cite{matouvsek2013determinant} who showed that the ratio between hereditary discrepancy and the determinant lower bound is bounded from above as
\[\frac{\herdisc(\mm{A})}{\detlb(\mm{A})} \lesssim \log(2mn)\cdot\sqrt{\log(2n)}.\]
Here we used the notation $a \lesssim b$ to denote the the existence of a universal constant $c$ such that $a \leq c\cdot b$. Similarly, $a \gtrsim b$ denotes the existence of a universal constant $c > 0$ such that $a \geq c\cdot b$, and $a\cong b$ when $a\lesssim b$ and $a\gtrsim b$.

Matou\v{s}ek's bound was not believed to be tight as the largest known
value of \(\frac{\herdisc(\mm{A})}{\detlb(\mm{A})}\) is on the order
of \(\log n\). Both the large discrepancy three
permutations family of~Newman, Neiman, and
Nikolov~\cite{newman2012beck}  (see also \cite{franks2018simplified})
and a construction due to P{\'a}lv{\"o}lgyi~\cite{palvolgyi2010indecomposable} achieve this gap.
 Matou\v{s}ek's bound follows from a pair of inequalities
\begin{align}
    \herdisc(\mm{A}) \lesssim \log(2mn)\cdot\hervecdisc(\mm{A})\label{eq:matousek-eq1},\\
    \hervecdisc(\mm{A}) \lesssim \sqrt{\log 2n}\cdot\detlb(\mm{A}),\label{eq:matousek-eq2}
\end{align}
where the first inequality is implied by the seminal result of Bansal~\cite{bansal2010constructive}, and the second inequality is proved using duality. 
Note that $\vecdisc$ is the \emph{vector discrepancy} of a set system, or matrix. This quantity is similar to discrepancy but the elements of the universe are ``colored'' by vectors rather than by $\pm 1$. In particular, for an ${m\times n}$ matrix $\mm{A}$, 
\[\vecdisc(\mm{A}) = \min_{\mm{v}_1, \ldots, \mm{v}_n \in S^{n-1}}\max_{j \in [m]}\left\|\sum_{i \in [n]} A_{j,i}\cdot \mm{v}_i\right\|_2,
\]
where $S^{n-1}$ is the unit sphere in $\RR^n$.
Note that vector discrepancy is a lower bound on discrepancy since a coloring $\chi: X \rightarrow \{\pm 1\}$ can be interpreted as a set of vectors where all vectors are parallel to each other.
The \emph{hereditary vector discrepancy of $\mm{A}$}, denoted $\hervecdisc$, which appears in \eqref{eq:matousek-eq1}, is the maximum vector discrepancy of any subset of the columns of $\mm{A}$. 

Recently Jiang and Reis~\cite{jiang2022tighter} were able to improve Matou\v{s}ek's result by showing that $\frac{\herdisc(\mm{A})}{\detlb(\mm{A})} \lesssim \sqrt{\log 2m\log 2n}$. Their work left open whether the \(\sqrt{\log m}\) term can be replaced by \(\sqrt{\log n}\) for large \(m\).
In the present work, we show that this is mostly not possible, and that the factor of $\sqrt{\log m}$ is necessary for all $m$ in the range $n \leq m \leq 2^{n^{1-\epsilon}}$ for any constant $\epsilon > 0$. Note that when $m > 2^n$, the matrix contains duplicated rows whose removal will not change the value of $\herdisc(\mm{A})$ nor $\detlb(\mm{A})$. Thus, our lower bound covers nearly the whole range of values for $m$.

Our main result is stated in the next theorem. Its proof appears at the end of Section~\ref{sec:haarbasis}.
\begin{theorem}\label{thm:main}
For any real number $\epsilon \in (0, 1)$, any integers $n \ge 2$ and $m \in \left[n, 2^{n^{1-\epsilon}}\right]$, there exists a matrix $\mm{A} \in \{0, 1\}^{m\times n}$ such that
\begin{equation}\label{eq:main}
    \frac{\herdisc(\mm{A})}{\detlb(\mm{A})} \gtrsim \sqrt{\log m \log n}.
\end{equation}
\end{theorem}

\modified{Note that the lower bound in Theorem~\ref{thm:main} only holds for $m \le 2^{n^{1-\epsilon}}$ for an arbitrarily small but fixed constant $\varepsilon$. This leaves open whether such a lower bound holds all the way to $m = 2^n$. The next theorem gives a new upper bound on $\herdisc(\mm{A})$ in terms of $\detlb(\mm{A})$, which implies that Theorem~\ref{thm:main} cannot be extended to $m = 2^{\omega(n/\log n)}$.
\begin{theorem}\label{thm:herdisc-ub}
    For all positive integers $m$ and $n$, and all matrices $\mm{A} \in \RR^{m\times n}$, we have 
    \[
    \frac{\herdisc(\mm{A})}{\detlb(\mm{A})} \lesssim \sqrt{n}.
    \]
\end{theorem}
This upper bound is based on the relationship between the volume lower bound on discrepancy studied in~\cite{dadush2018balancing}, and the determinant lower bound. In particular, we show that the volume lower bound is bounded by a constant multiple of $\sqrt{n}\cdot \detlb(\mm{A})$, and use a result from~\cite{dadush2018balancing} characterizing the hereditary discrepancy of partial colorings in terms of the volume lower bound. We also give a simpler proof for the special case of \(\mm{A} \in \{0,1\}^{m\times n}\) in the Appendix, using the theory of VC dimension.
}

Crucial to the proof of Theorem~\ref{thm:main} is a recursively
defined $2^{k} \times 2^k$ matrix $\mm{A}_k$, based on the Haar
wavelet basis. We will also show that $\mm{A}_k$ is tight for Equation~\eqref{eq:matousek-eq2}. 
\begin{theorem}\label{thm:tight-matousek-example}
    With $n = 2^k$ for an integer $k\ge 1$, $\hervecdisc(\mm{A}_k) \gtrsim \sqrt{\log n}\cdot\detlb(\mm{A}_k)$.
\end{theorem}
Its proof appears in Section~\ref{sec:othernotableproperties}. It is
yet unknown whether Equation~\eqref{eq:matousek-eq1} is tight as Jiang
and Reis improved upon Matou\v{s}ek's bound by circumventing the
inequality altogether. The resolution of this problem via an efficient
algorithm would imply new and old constructive bounds, for example,
the constructive version of Banaszczyk's upper bound for the
Beck-Fiala problem~\cite{banaszczyk1998balancing,bansal2016algorithm},
and a constructive version of Nikolov's upper bound for Tu\'snady's
problem~\cite{nikolov17tusnady}.

Our use of the matrix $\mm{A}_k$ is inspired by work of
Kunisky~\cite{kunisky2023discrepancy}, \modified{who first used this matrix in the context of proving discrepancy lower bounds}. We also observe that this
matrix gives easier proofs of other known results in discrepancy
theory: see Section~\ref{sec:othernotableproperties}.

Before we begin the proof proper, we define a few variants of discrepancy which will appear throughout this exposition. Just as $\disc(\mm{A})$ was defined in terms of $L_{\infty}$, we can define discrepancy in terms of other norms. In particular, for $L_1$,
\[\disc_1(\mm{A}) \coloneqq \min_{\mm{x} \in \{\pm 1\}^n}\frac{\norm{\mm{A}\mm{x}}_1}{m}\]
and generally for $L_p$,
\[\disc_p(\mm{A}) \coloneqq \min_{\mm{x} \in \{\pm 1\}^n} \left(\frac{\norm{\mm{A}\mm{x}}_p^p}{m}\right)^{1/p}.\]
Note that $\disc_p(\mm{A}) \leq \disc_q(\mm{A})$ when $p \leq q$.

\section{Proof Structure}
In order to prove Theorem~\ref{thm:main} we will find a family of
matrices which satisfy equation~\eqref{eq:main}. Our candidates will
have the form $\mm{P}_N \otimes \mm{A}$ where $\mm{P}_N$ is the
$2^{N}\times N$ incidence matrix of the power set, and $\mm{A}$ is some
$p \times p$ matrix with a gap between $\detlb(\mm{A})$ and
$\disc_1(\mm{A})$. \modified{In particular, we let $\mm{A}$ be the Haar
basis matrix used in the work of Kunisky~\cite{kunisky2023discrepancy}, and prove some properties of $\mm{A}$ in-order to obtain the present result.}

We bound $\detlb(\mm{P}_N \otimes \mm{A})$ from above by showing that $\detlb(\mm{P}_N \otimes \mm{A}) \lesssim \sqrt{N}\cdot\detlb(\mm{A})$ using standard linear algebra and Lemma 4 from \cite{matouvsek2013determinant}. See Lemma~\ref{lem:detlb-amplification}. For our choice of $\mm{A}$, we will show that $\detlb(\mm{A}) \lesssim 1$. We also bound $\disc(\mm{P}_N \otimes \mm{A}) \gtrsim N\cdot \disc_1(\mm{A})$ using a discrepancy amplification argument. See Lemma~\ref{lem:disc-amplification}. By finding a tight lower bound on $\disc_1(\mm{A})$, we obtain the lower bound $\disc(\mm{P}_N \otimes \mm{A}) \gtrsim N\cdot\sqrt{\log p}$. Taken together, these bounds gives us a gap on the order of $\sqrt{N}\cdot \sqrt{\log p}$ between $\detlb(\mm{P}_N \otimes \mm{A})$ and $\disc(\mm{P}_N \otimes \mm{A})$.

\begin{lemma}\label{lem:detlb-amplification} For the power matrix $\mm{P}_N$, and any real matrix $\mm{A}$, 
    \[\detlb\left(\mm{P}_N \otimes \mm{A}\right)\leq \sqrt{eN}\cdot\detlb(\mm{A}).\] 
\end{lemma}
\begin{proof}
    Let $\mm{u}_1, ..., \mm{u}_N$ be the columns of $\mm{P}_N$ and $\mm{A} \in \RR^{p\times p}$. Divide the columns of $\mm{P}_N \otimes \mm{A}$ into $N$ contiguous blocks of size $(2^{N}p)\times p$ each representing $\mm{u}_\ell \otimes \mm{A}$. Note that $\mm{u}_\ell \otimes \mm{A}$ consists of $2^N$ blocks of $\mm{A}$ or $\mm{0}$ stacked on top of one another. We claim that $\detlb(\mm{u}_\ell \otimes \mm{A}) \leq \detlb(\mm{A})$. Consider an $s\times s$ sub-matrix $\mm{B}$ of $\mm{u}_\ell \otimes \mm{A}$ with the rows indexed by $I$ and columns indexed by $J$. Note that if any row of $\mm{B}$ is zero, then $\det(\mm{B}) = 0$ so in order for the determinant to be non-zero, the rows of $\mm{B}$ must be parts of rows of $\mm{A}$ with columns indexed by $J$. If there are multiple copies of the same row of $\mm{A}$, then again $\det(\mm{B}) = 0$. Thus $\mm{B}$ must come from distinct rows of $\mm{A}$ with columns indexed by $J$. It follows that $\mm{B}$ is actually a sub-matrix of $\mm{A}$ up to rearrangement of the rows, so $\left|\det(\mm{B})\right|^{1/s} \leq \detlb(\mm{A})$. Since this is true for all choices of the submatrix $\mm{B}$, we have $\detlb\left(\mm{u}_\ell \otimes \mm{A}\right) \leq \detlb(\mm{A})$.

    Recall from \cite{matouvsek2013determinant} Lemma 4, that for real matrices $\mm{B}_1, ..., \mm{B}_t$ each with the same number of columns and $D\coloneqq \max_{i=1,2,...t}\detlb(\mm{B}_i)$, any matrix $\mm{B}$ whose rows are copies of the rows of the matrices $\mm{B}_i$ satisfies $\detlb(\mm{B}) \leq D\sqrt{e t}$. By applying Lemma 4 to $\left(\mm{P}_N \otimes \mm{A}\right)^{\top}$ with $\mm{B}_i = \left(\mm{u}_i \otimes \mm{A}\right)^{\top}$, and we have that
    \begin{equation*}
      \detlb(\mm{P}_N \otimes \mm{A}) \leq \sqrt{e
      N}\cdot\max_{\ell \in [N]}\detlb(\mm{u}_\ell \otimes \mm{A})
      \leq \sqrt{e N}\cdot\detlb(\mm{A}).
    \end{equation*}\qed
  \end{proof}

\begin{lemma}{\textup{(Discrepancy Amplification)}.}\label{lem:disc-amplification}
For the power matrix $\mm{P}_N$ and any real matrix $\mm{A}$, 
\[\disc(\mm{P}_N \otimes \mm{A}) \geq \frac{N\cdot \disc_1(\mm{A})}{2}.\]
\end{lemma}
\begin{proof}
Let $\mm{A} \in \RR^{p\times q}$ and $t \coloneqq \disc_1(\mm{A})$. Consider some vector $\mm{x}\in \{\pm 1\}^{qN}$ composed of vectors $\mm{x}^{(1)}, ..., \mm{x}^{(N)}$ stacked on top of each other containing $p$ entries each. We compute $\norm{(\mm{P}_N \otimes \mm{A})\mm{x}}_{\infty}$. Note that
\[\norm{(\mm{P}_N \otimes \mm{A})\mm{x}}_{\infty} = \max_{S\subseteq [N]}\left\|\sum_{i\in S}\mm{A}\mm{x}^{(i)}\right\|_{\infty} = \max_{S\subseteq[N]}\max_{j \in [p]}\left|\sum_{i \in S}\left(\mm{A}\mm{x}^{(i)}\right)_j\right|.\]
From the assumption, we have $\frac{1}{p}\norm{\mm{A}\mm{x}^{(i)}}_1 \geq t$ for every $i \in [N]$. Taking an average over all choices of $i$, 
\[t \le \frac{1}{pN}\sum_{i=1}^N\sum_{j=1}^p\left|\left(\mm{A}\mm{x}^{(i)}\right)_j\right| = \frac{1}{pN}\sum_{j=1}^p\sum_{i=1}^N\left|\left(\mm{A}\mm{x}^{(i)}\right)_j\right| \implies \sum_{i=1}^N\left|\left(\mm{A}\mm{x}^{(i)}\right)_j\right| \geq Nt\]
for some $j \in [p]$. With $S^{+} = \{i: (\mm{A}\mm{x}^{(i)})_j > 0\}$ and $S^{-} = \{i: (\mm{A}\mm{x}^{(i)})_j < 0\}$, 
\begin{align*}
    Nt \leq \sum_{i=1}^{N}\left|\left(\mm{A}\mm{x}^{(i)}\right)_j\right| &= \sum_{i\in S^+} \left(\mm{A}\mm{x}^{(i)}\right)_j - \sum_{i \in S^-} \left(\mm{A}\mm{x}^{(i)}\right)_j\\
    &= \left|\sum_{i\in S^+} \left(\mm{A}\mm{x}^{(i)}\right)_j\right| + \left|\sum_{i \in S^-} \left(\mm{A}\mm{x}^{(i)}\right)_j\right|\\
    &\implies \max\left\{\left|\sum_{i\in S^+} \left(\mm{A}\mm{x}^{(i)}\right)_j\right|,  \left|\sum_{i \in S^-} \left(\mm{A}\mm{x}^{(i)}\right)_j\right| \right\} \geq \frac{Nt}{2}.
\end{align*}
\end{proof}

Lemmas~\ref{lem:detlb-amplification}~and~\ref{lem:disc-amplification} together imply that
\[
\frac{\herdisc(\mm{P}_N \otimes \mm{A})}{\detlb(\mm{P}_N \otimes \mm{A})} \geq \frac{\sqrt{N}}{2\sqrt{e}} \cdot \frac{\disc_1(\mm{A})}{\detlb(\mm{A})}.
\]
Note that, if \(\mm{A} \in \RR^{p\times q}\), then \(\mm{P}_N \otimes
\mm{A}\) is an \((2^N p)\times (Nq)\) matrix, so \(\sqrt{N}\) is
roughly \(\sqrt{\log m}\) for small enough \(p\), where \(m \coloneqq
2^N p\) is the number of rows of \(\mm{P}_N \otimes \mm{A}\). To prove
Theorem~\ref{thm:main}, we need to find a matrix \(\mm{A}\) that
exhibits a large gap between \(\disc_1(\mm{A})\) and
\(\detlb(\mm{A})\). In the next section, we show that a matrix whose
columns are the discrete Haar basis vectors has this property.

\section{Discrete Haar Basis}\label{sec:haarbasis}
The $2^k\times 2^k$ discrete Haar basis matrix $\mm{A}_k$ is defined recursively with $\mm{A}_0 = [1]$ and 
\begin{equation}\label{eq:haardef}
\mm{A}_k = \begin{bmatrix}\mm{A}_{k-1} & \mm{I}_{2^{k-1}}\\\mm{A}_{k-1} & -\mm{I}_{2^{k-1}} \end{bmatrix},    
\end{equation}
where $\mm{I}_{2^{k-1}}$ is the $2^{k-1} \times 2^{k-1}$ identity matrix. This matrix arises from the following tree structure. Construct a depth $k$ perfect binary tree, and let $r$ be an additional node. We make the root of the perfect binary tree the left child of $r$, and $r$ becomes the root of our tree. Every non-leaf node represents a column in the matrix while every root-to-leaf path corresponds to a row in the matrix. Whenever the path proceeds down the left child from some node $i$, entry $i$ of the corresponding row will have value $+1$. If instead the path proceeds down the right child of $i$, entry $i$ of the corresponding row will have value $-1$. Thus every row will have $k$ non-zero entries. It is also not hard to show that, for any $\pm 1$ coloring of the columns, there is a row whose nonzero entries are equal to the corresponding column colors, and, therefore, $\disc(\mm{A}_k) = k$. Kunisky~\cite{kunisky2023discrepancy} describes this in detail.

In addition, we define the $\{0,1\}^{2^k\times 2^k}$ matrices $\mm{A}_k^+$ and $\mm{A}_k^-$ to be the indicator matrices of the positive and negatives elements of $\mm{A}_k$ respectively. Here an \emph{indicator matrix} will have one in some entry if and only the corresponding entry of $\mm{A}_k$ is non-zero and positive, in the case of $\mm{A}_k^+$, or negative, in the case of $\mm{A}_k^-$. Note that $\mm{A}_k = \mm{A}_k^{+} - \mm{A}_k^{-}$. Finally define, 
\begin{equation}\label{eq:haar-indicators}
    \mm{A}_k^{\pm}\coloneqq \begin{pmatrix}\mm{A}_k^{+}\\\mm{A}_k^{-}\end{pmatrix}.
\end{equation}
We bound the hereditary discrepancy to determinant lower bound ratio for both $\mm{P}_N \otimes \mm{A}_k$ and $\mm{P}_N \otimes \mm{A}_k^{\pm}$.

\begin{theorem}\label{thm:kunisky-ratio}
    For the power matrix $\mm{P}_N$, the discrete Haar basis $\mm{A}_k$, and the stacked indicator matrix $\mm{A}_k^{\pm}$ as defined in equation~\eqref{eq:haar-indicators},
    \begin{align}
        \frac{\herdisc\left(\mm{P}_{N}\otimes\mm{A}_k\right)}{\detlb\left(\mm{P}_{N}\otimes\mm{A}_k\right)} &\gtrsim \sqrt{N\cdot k}, \label{eq:kunisky-ratio}\\
        \frac{\herdisc\left(\mm{P}_{N}\otimes\mm{A}_k^{\pm}\right)}{\detlb\left(\mm{P}_{N}\otimes\mm{A}_k^{\pm}\right)} &\gtrsim \sqrt{N\cdot k} \label{eq:haar-indicator-ratio}.
    \end{align}
\end{theorem}
\begin{proof}
First we apply the proof structure described in the previous section to $\mm{A}_k$. In particular, we show that $\detlb(\mm{A}_k) = O(1)$ in Lemma~\ref{lem:ub-kunisky-detlb} and that $\disc_1(\mm{A}_k) \gtrsim \sqrt{k}$ in Lemma~\ref{lem:lb-kunisky-disc1}. Applying Lemma~\ref{lem:detlb-amplification} to the first result and Lemma~\ref{lem:disc-amplification} to the second, we have that $\detlb\left(\mm{P}_{N}\otimes\mm{A}_k\right) \lesssim \sqrt{N}$ and $\disc\left(\mm{P}_{N}\otimes\mm{A}_k\right) \gtrsim N\cdot\sqrt{k}$. It follows that 
\[\frac{\herdisc\left(\mm{P}_{N}\otimes\mm{A}_k\right)}{\detlb\left(\mm{P}_{N}\otimes\mm{A}_k\right)} \geq \frac{\disc\left(\mm{P}_{N}\otimes\mm{A}_k\right)}{\detlb\left(\mm{P}_{N}\otimes\mm{A}_k\right)} \gtrsim \sqrt{N\cdot k}.\]   
The process for $\mm{A}_k^{\pm}$ is similar. To show an upper bound on $\detlb(\mm{A}_k^{\pm})$, use Corollary~\ref{cor:haar-decomp} where $\mm{A}_k^+$ and $\mm{A}_k^-$ are shown to be TUM. Since the determinant of any square submatrix of either matrix is at most one in absolute value, we can apply Lemma 4 of~\cite{matouvsek2013determinant} to $\mm{A}_k^+$ and $\mm{A}_k^-$ to obtain $\detlb(\mm{A}_k^{\pm}) = O(1)$. To obtain the lower bound on $\disc_1(\mm{A}_k^{\pm})$, we will recall that $\disc_1(\mm{A}_k) \gtrsim \sqrt{k}$ from Lemma~\ref{lem:lb-kunisky-disc1}. Note that, for any \(\mm{x} \in \{-1, +1\}^{2^k}\), by the triangle inequality
\[
\frac{1}{2^k}\|\mm{A}_k\mm{x}\|_1 
= \frac{1}{2^k}\|(\mm{A}_k^+ - \mm{A}_k^-) \mm{x}\|_1 
\le 2\left(\frac{1}{2^{k+1}}\|\mm{A}_k^+\mm{x}\|_1 + \frac{1}{2^{k+1}}\|\mm{A}_k^-\mm{x}\|_1\right).
\]
Therefore, $\disc_1(\mm{A}_k^{\pm}) \gtrsim \sqrt{k}$ as well. Apply Lemma~\ref{lem:detlb-amplification} and Lemma~\ref{lem:disc-amplification} to $\detlb(\mm{A}_k^{\pm}) = O(1)$ and $\disc_1(\mm{A}_k^{\pm 1}) \gtrsim \sqrt{k}$ respectively to obtain equation~\eqref{eq:haar-indicator-ratio}.
\end{proof}

\begin{lemma}\label{lem:ub-kunisky-detlb}
    $\detlb(\mm{A}_k) \leq 2$.
\end{lemma}
\begin{proof}
We show that any $i \times i$ square submatrix $\mm{B}$ of $\mm{A}_k$ satisfies $|\det(\mm{B})| \leq 2^{i}$. First, define $M_{k}(i) \coloneqq \max_{\mm{B}}|\det(\mm{B})|$ where the maximum is taken over all $i\times i$ submatrices $\mm{B}$ of $\mm{A}_k$. We compute $M_{k}(i)$ recursively by considering the forms that all $i\times i$ submatrices of $\mm{A}_k$ can take:
\begin{enumerate}
    \item $\mm{B}$ only contain elements from the first $2^{k-1}$ columns of $\mm{A}_k$,
    \item $\mm{B}$ only contain elements from the second $2^{k-1}$ columns of $\mm{A}_k$, or
    \item $\mm{B}$ contain elements from both the first and second $2^{k-1}$ columns of $\mm{A}_k$.
\end{enumerate}
We use the recursive formula \eqref{eq:haardef} to analyze these cases.
In the first case the resulting submatrix is either entirely contained in $\mm{A}_{k-1}$ up to rearranging rows, or contains a duplicated row. The magnitude of the determinant of these submatrices can be bounded above by $M_{k-1}(i)$ and $0$, respectively. In the second case we note that the submatrix is TUM. To see this, recall that the second $2^{k-1}$ columns of $\mm{A}_k$ consist of an identity matrix and its negation stacked on top of one another. Any square submatrix is entirely contained in the identity matrix or contains duplicated (and negated) rows. Thus the absolute value of the determinant of this kind of submatrix is at most one. It remains to consider the third case. Let $\mm{B}$ be a submatrix of $\mm{A}_k$ with some $j$ columns coming from the second $2^{k-1}$ columns of $\mm{A}_k$ for $1 \leq j < i$. For any such column there is either one or two non-zero entries, equal to $1$ or $-1$. If there is only one non-zero entry, then this reduces to computing $M_k(i-1)$ since we can perform a co-factor expansion on this column. If there are two non-zero entries, then we can permute the rows so that they are adjacent. This only changes the sign of the resulting determinant. Notice that the two rows are identical except for the sign of the non-zero entries. When performing a co-factor expansion on these two entries, the $(i-1)\times(i-1)$ submatrix that results when removing either row and the column is identical. Thus $|\det(\mm{B})|$ is at most twice the absolute value of the determinant of this $(i-1)\times(i-1)$ submatrix. After removing all the columns and associated rows of $\mm{B}$ from the second half of $\mm{A}_k$ in this way, we see that $|\det(\mm{B})| \leq 2^{j}M_{k-1}(i-j)$. Since, in the base case, $M_0(1) = 1$, the claim follows.
\end{proof}
Using a similar argument as above, we can show that the matrices $\mm{A}_k^+$ and $\mm{A}_k^-$ are TUM. Note that, using \eqref{eq:haardef}, $\mm{A}_k^+$ and $\mm{A}_k^-$ can be recursively defined as $\mm{A}_0^+ = [1]$, $\mm{A}_0^- = [0]$, and
\begin{equation}\label{eq:recursive-def-signed-id}
    \mm{A}_k^+ = \begin{bmatrix}
        \mm{A}_{k-1}^+ & \mm{I}_{2^{k-1}}\\
        \mm{A}_{k-1}^+ & \mm{0}
    \end{bmatrix}
    \qquad
    \mm{A}_k^- = \begin{bmatrix}
        \mm{A}_{k-1}^- & \mm{0}\\
        \mm{A}_{k-1}^- & \mm{I}_{2^{k-1}}
    \end{bmatrix},
\end{equation}
where $\mm{0}$ is the all zeros matrix of appropriate dimension, and $\mm{I}_{2^{k-1}}$ is the $2^{k-1}$ by $2^{k-1}$ identity matrix.

\begin{corollary}\label{cor:haar-decomp}
    $\mm{A}_k^+$ and $\mm{A}_k^-$ are TUM matrices where $\mm{A}_k^+$ and $\mm{A}_k^-$ are indicators of the positive and negatives entries of $\mm{A}_k$, respectively.
\end{corollary}
\begin{proof}
    We only consider $\mm{A}_k^+$ as the proof that $\mm{A}_k^-$ is a TUM matrix is similar. The proof proceeds by induction on $k$. Consider some $i \times i$ submatrix $\mm{B}$ of $\mm{A}_k^+$. If $\mm{B}$ is entirely contained in first half of the columns of $\mm{A}_k^+$, then we are done by the inductive hypothesis; if $\mm{B}$ is entirely contained in the second half of the columns of $\mm{A}_k^+$, then $\mm{B}$ is a submatrix of the identity matrix or has a row of $0$'s, and the absolute value of its determinant is at most $1$. Thus it suffices to consider the case where $\mm{B}$ has some columns from the first half of $\mm{A}_k^+$ and some columns from the second half of $\mm{A}_k^+$. Since any column from the second half has only one non-zero entry, equal to $1$, performing  co-factor expansions on the columns in the second half shows that the absolute value of the determinant will only be as large as the absolute value of the determinant of some smaller square sub-matrix in $\mm{A}_{k-1}^+$. Note that in the base case, $|\det(\mm{A}_0^+)| = 1$.
\end{proof}

\begin{lemma}\label{lem:lb-kunisky-disc1}
$\disc_1(\mm{A}_k) = \frac{k+1}{2^{k}}\binom{k}{\floor{(k+1)/2}} \cong \sqrt{k}$.
\end{lemma}
\begin{proof}
The proof for $k = 0$ is trivial, so we focus on the case $k \ge 1$. Let $\tilde{\mm{A}}_k$ denote the $2^k \times (2^k - 1)$ matrix equal to $\mm{A}_k$ with the first column, all of whose entries are $1$, removed. Note that this is equivalent to removing the root node $r$ and keeping only the perfect binary tree of depth $k$ in the tree structure of the Haar basis, as described at the beginning of the section. \modified{We have the following key claim.}

\begin{claim}\label{claim:haar-permute}
    For any $\mm{x} \in \{\pm 1\}^{2^k - 1}$ there exists a permutation which maps the entries of $\tilde{\mm{A}}_k\mm{x}$ to those of $\tilde{\mm{A}}_k\mathbf{1}$. 
\end{claim}
\begin{proof}
Our proof is by induction on $k$. When $k = 1$,we have a root node with two children corresponding to the matrix
\[
\tilde{\mm{A}}_1 = \begin{bmatrix}
1\\
-1
\end{bmatrix}.
\]
 When $\mm{x} = [1]$, $\tilde{\mm{A}}_1\mm{x} = \tilde{\mm{A}}_1\mathbf{1}$ and the identity permutation suffices; when $\mm{x} = [-1]$, $\tilde{\mm{A}}_1\mm{x} = -\tilde{\mm{A}}_1\mathbf{1}$ and it suffices to swap the two entries.

Consider some height $k$ perfect binary tree corresponding to $\tilde{\mm{A}}_k$. Let $u$ be the root of the tree with left and right children $u_+$ and $u_-$ respectively. Since every root-to-leaf path must go through $u_+$ or $u_-$, this forms a partition of the rows of $\tilde{\mm{A}}_k$. In particular, we can rearrange $\tilde{\mm{A}}_k$ as
\[\tilde{\mm{A}}_k = \begin{bmatrix}
    \mathbf{1} & \tilde{\mm{A}}_{k-1} & \mm{0}\\
    -\mathbf{1} & \mm{0} & \tilde{\mm{A}}_{k-1} 
\end{bmatrix}.\]
Consider the $\tilde{\mm{A}}_{k-1}$ submatrices which appear in $\tilde{\mm{A}}_{k}$. The $\tilde{\mm{A}}_{k-1}$ submatrix in the first $2^{k-1}$ rows has rows which correspond to root-to-leaf paths with leaves in the subtree rooted at $u_+$. Its columns correspond to nodes in the same subtree. The $\tilde{\mm{A}}_{k-1}$ submatrix in the second $2^{k-1}$ rows is defined similarly on the subtree rooted at $u_-$. Write the vector $\mm{x}$ as $[x_u, \mm{x}_+, \mm{x}_-]^{\top}$ where $x_u$ is the color of the node $u$ and $\mm{x}_+$ and $\mm{x}_-$ are the colors of the nodes in the subtrees rooted at $u_+$ and $u_-$, respectively. Consider the value of $x_u$. If $x_u = 1$, then by the inductive hypothesis, there exists a permutation which takes the entries of $\tilde{\mm{A}}_{k-1}\mm{x}_+$ to the entries of $\tilde{\mm{A}}_{k-1}\mathbf{1}$ and another permutation which takes the entries of $\tilde{\mm{A}}_{k-1}\mm{x}_-$ to the entries of $\tilde{\mm{A}}_{k-1}\mathbf{1}$. These two permutations can be combined to form a permutation which maps the entries of $\tilde{\mm{A}}_k\mm{x}$ to the entries of $\tilde{\mm{A}}_k\mathbf{1}$. Otherwise $x_u = -1$. Again, there exists a permutation $\pi_1$ which takes $\tilde{\mm{A}}_{k-1}\mm{x}_+$ to $\tilde{\mm{A}}_{k-1}\mathbf{1}$ and another permutation $\pi_2$ which takes $\tilde{\mm{A}}_{k-1}\mm{x}_-$ to $\tilde{\mm{A}}_{k-1}\mathbf{1}$. We can construct a permutation which maps the elements of $\tilde{\mm{A}}_k\mm{x}$ to those of $\tilde{\mm{A}}_k\mathbf{1}$ by by first applying $\pi_1$ to the first $2^{k-1}$ entries of $\tilde{\mm{A}}_k\mm{x}$, and $\pi_2$ to the remaining $2^{k-1}$ entries, and then swapping the first $2^{k-1}$ entries with the second $2^{k-1}$ entries. 
\end{proof}
\modified{
Note that, for any $\mm{x} \in\{\pm 1\}^{2^k}$, \(\disc_1(\mm{A}_k,\mm{x}) = \disc_1(\mm{A}_k,-\mm{x})\). Then, we have
\[
\disc_1(\mm{A}_k,\mm{x}) = \frac12(\disc_1(\mm{A}_k,\mm{x}) + \disc_1(\mm{A}_k,-\mm{x}))
\]
By Claim~\ref{claim:haar-permute}, and, since all entries of the first column of $\mm{A}_k$ are equal to $1$,
\begin{align*}
    \disc_1(\mm{A}_k,\mm{x}) &= \frac{1}{2^{k}} \sum_{i=1}^{2^k}\frac{|\tilde{\mm{a}}_i^\top \mathbf{1}+1| + |\tilde{\mm{a}}_i^\top \mathbf{1}-1|}{2}\\
    &= \frac{1}{2^{k+1}} \sum_{i=1}^{2^k}|\tilde{\mm{a}}_i^\top \mathbf{1}+1| 
    + 
    \frac{1}{2^{k+1}} \sum_{i=1}^{2^k}  |\tilde{\mm{a}}_i^\top \mathbf{1}-1|,
\end{align*}
where $\tilde{\mm{a}}_i^\top$ is the $i$-th row of $\tilde{\mm{A}}_k$. Recall that each row of $\tilde{\mm{A}}_k$ has exactly $k$
non-zero entries, and every sign pattern for these $k$ entries appears exactly once, as each row in $\tilde{\mm{A}}_k$ corresponds to a root-to-leaf path in a depth $k$ perfect binary tree, and the sign pattern corresponds to the sequence of left and right turns made by the path. In particular, there are exactly ${k\choose \ell}$ rows $\tilde{\mm{a}}_i$ for which $\tilde{\mm{a}}_i^\top \mathbf{1}$ equals $k-2\ell$, since such rows have $k-\ell$ non-zero entries equal to $+1$, and $\ell$ non-zero entries equal to $-1$. 
Substituting above, we have
\begin{align*}
   \frac{1}{2^{k+1}} \sum_{i=1}^{2^k}|\tilde{\mm{a}}_i^\top \mathbf{1}+1| 
    + 
    &\frac{1}{2^{k+1}} \sum_{i=1}^{2^k} |\tilde{\mm{a}}_i^\top \mathbf{1}-1|\\
    &=
     \frac{1}{2^{k+1}} \sum_{\ell=0}^k {k\choose \ell}|k+1-2\ell| 
    + 
    \frac{1}{2^{k+1}} \sum_{\ell=0}^k {k\choose \ell}|k-1-2\ell|\\
    &= 
     \frac{1}{2^{k+1}} \sum_{\ell=0}^{k} {k\choose \ell}|k+1-2\ell| 
    + 
    \frac{1}{2^{k+1}} \sum_{\ell=1}^{k+1} {k\choose \ell-1}|k+1-2\ell|\\
     &=
    \frac{1}{2^{k+1}} \sum_{\ell=0}^{k+1} \left({k\choose \ell}+{k\choose \ell-1}\right)|k+1-2\ell| \\
    &=
    \frac{1}{2^{k+1}} \sum_{\ell=0}^{k+1} {k+1\choose \ell}|k+1-2\ell|.
\end{align*}
The second equality above follows by a change of variables in the second sum, the third equality uses the convention $\binom{k}{k+1} = \binom{k}{-1} = 0$, and the last equality follows form Pascal's identity.}

\modified{Now, by Lemma~\ref{lem:lb-kunisky-computation} below, we have that \[\sum_{\ell=0}^{k+1} {k+1\choose \ell}|k+1-2\ell| = 2(k+1)\binom{k}{\floor{(k+1)/2}},\] which implies that $\disc_1(\tilde{\mm{A}}_k) \cong\sqrt{k}$ since $\binom{k}{\floor{(k+1)/2}} \cong 2^{k}/\sqrt{k}$ by Stirling's approximation.}
\end{proof}

The above proof also has a probabilistic interpretation. We show that, for a uniformly random row $\mm{a}^{\top}$ of $\tilde{\mm{A}}_k$ and a fixed coloring $\mm{x} \in \{\pm 1\}^{2^k-1}$, $\mm{a}^{\top}\mm{x}$ is distributed like $X_1 + \cdots + X_k$ where the $X_i$s are independent Rademacher random variables \modified{(i.e. random variables uniform in $\{-1,+1\}$)}. Recall that uniformly choosing a row of $\tilde{\mm{A}}_k$ corresponds to uniformly choosing a root-to-leaf path in the depth $k$ perfect binary tree. Further, the non-leaf nodes of the tree correspond to the columns of $\tilde{\mm{A}}_k$. Thus we know that exactly $k$ entries of the row will be non-zero. \modified{Let the indices of these entries be $U_1, ..., U_k$, and note that $\mm{a}^{\top}\mm{x} = x_{U_1}a_{U_1} + \cdots + x_{U_k}a_{U_k}$. The key observation is that, conditional on the values of $U_1, \ldots, U_{\ell}$, $a_{U_\ell}$ is equally likely to be $-1$ or $+1$, since the a uniformly random path in the binary tree going through $U_1, \ldots, U_\ell$ is equally likely to visit the left or the right child of $U_\ell$. We can then show that $x_{U_1}a_{U_1} + \cdots + x_{U_k}a_{U_k}$ has the same distribution as a sum of $k$ independent Rademacher random variables by induction on $k$. In the base case, $a_{U_1}$ is uniform in $\{-1, +1\}$, and so is $x_{U_1}a_{U_1}$. Suppose $x_{U_1}a_{U_1} + \cdots + x_{U_{k-1}}a_{U_{k-1}}$ is distributed as the sum of $k-1$ independent Rademacher random variables. Conditional on the choice of $U_1, \ldots, U_k$, $a_{U_k}$, and, therefore, $x_{U_k}a_{U_k}$ are equally likely to be $-1$ or $+1$. Taking expectation over the choice of $U_1, \ldots, U_k$ finishes the proof.}

The next lemma is likely a well-known calculation. We include a proof due to Lavrov, for completeness.
\begin{lemma}[\cite{lavrov2018}]\label{lem:lb-kunisky-computation}
    $\sum_{\ell=0}^{k}\binom{k}{\ell}|k-2\ell| = 2k\cdot\binom{k-1}{\floor{k/2}}$.
\end{lemma}
\begin{proof}
Recall the identity $\binom{k}{\ell}\ell = \binom{k-1}{\ell-1}k$. We write
\begin{align*}
    \sum_{\ell = 0}^k\binom{k}{\ell}|k-2\ell| &= \sum_{\ell < k/2}\binom{k}{\ell}(k-2\ell) - \sum_{\ell > k/2}\binom{k}{\ell}(k-2\ell)\\
    &= k\left(\sum_{\ell < k/2}\binom{k}{\ell} - \sum_{\ell > k/2}\binom{k}{\ell}\right) - 2\left(\sum_{\ell < k/2}\binom{k}{\ell}\ell - \sum_{\ell > k/2}\binom{k}{\ell}\ell\right)\\
    &= 2k\left(\sum_{\ell > k/2}\binom{k-1}{\ell-1} - \sum_{\ell < k/2}\binom{k-1}{\ell-1}\right)\\
    &= 2k\binom{k-1}{\floor{k/2}}.
\end{align*}
Here, the first equality follows since $\binom{k}{\ell}\left(k - 2(k/2)\right) = 0$ when $k$ is even. The last equality follows by consider the parity of $k$; when $k$ is even, we obtain a $\binom{k-1}{k/2}$ term after cancellation, and when $k$ is odd, we obtain a $\binom{k-1}{\floor{k/2}}$ term after cancellation.
\end{proof}
Note that when we divide the identity by $2^k$, we obtain the expectation of a sum of $k$ independent Rademacher random variables. The asymptotic version of this identity follows from Khintchine's inequality.

\begin{proof}[Proof of Theorem~\ref{thm:main}]
    Consider the range of $m$ in $[n, n^2]$ and $\left(n^2, 2^{n^{(1-\epsilon)}}\right]$ separately. In the first interval, we let $\mm{A}$ be the matrix $\mm{A}_k$ padded with $m - n$ rows of zeros. Here, $\frac{\herdisc(\mm{A})}{\detlb(\mm{A})} \cong \log n \cong \sqrt{\log m\cdot \log n}$. When $m \in \left(n^2, 2^{n^{(1-\epsilon)}}\right]$, we consider the matrix $\mm{P}_N\otimes \mm{A}_k$ where $N = \floor{\log_2(m/n)}$ and $k = \floor{\log_2 n^{\epsilon}}$. Observe that $\mm{P}_{N}\otimes\mm{A}_{k}$ is an $m' \times n'$ matrix where $m' = 2^{N+k} \leq m/n^{1 - \epsilon} < m$ and $n' = N\cdot2^k \leq \log_2(m/n) \cdot n^{\epsilon} \leq n - n^{\epsilon}\log n$ since $m \leq 2^{n^{(1-\epsilon)}}$. We obtain $\mm{A}$ by padding $\mm{P}_N\otimes\mm{A}_k$ with zero vectors so that it has exactly $m$ rows and $n$ columns. Note that $\log m \cong N$ and $\log n \cong k$. By Theorem~\ref{thm:kunisky-ratio}, $\frac{\herdisc(\mm{A})}{\detlb(\mm{A})}  \gtrsim\sqrt{Nk} \gtrsim \sqrt{\log m\log n}$, as required. 

    The reader might object that the matrix $\mm{A}_k$ has negative entries which would not occur for incidence matrices of a set system, but we can remedy this by considering $\mm{A}_k^{\pm}$ as defined in Theorem~\ref{thm:kunisky-ratio} instead.
\end{proof}

\subsection{Other Notable Properties}\label{sec:othernotableproperties}
We can make a few more observations about the properties of $\mm{A}_k$.
\begin{claim}
    $|\det(\mm{A}_{k})| = 2^{2^k - 1}$.
\end{claim}
\begin{proof}
    Note that the columns of $\mm{A}_{k}$ are orthogonal so $|\det(\mm{A}_k)|$ is equal to the product of the $\ell_2$-norms of columns. Since the $i$th column has magnitude $2^{2^{i-1}}$, $|\det(\mm{A}_k)| = 2^{2^{k-1} + 2^{k-2} + \cdots + 2^{0}} = 2^{2^k-1}$.
\end{proof}

Next we prove Theorem~\ref{thm:tight-matousek-example}, showing that $\mm{A}_k$ serves as an example that equation~\eqref{eq:matousek-eq2} of~\cite{matouvsek2013determinant} --- mentioned in the introduction --- is tight. 
\begin{proof}[Proof of Theorem~\ref{thm:tight-matousek-example}]
To see this, it suffices to show that $\vecdisc\left(\mm{A}_k\right)^2 = \Omega(k)$. Let $\mm{v}_0, \mm{v}_1, ..., \mm{v}_q$ be the vector colors assigned to the $2^k$ columns of $\mm{A}_k$. Recall that $\mm{A}_k$ corresponds to a tree with root node $r$, where $r$ has no right child, and the left child is the root of a perfect binary tree of depth $k$. The root to leaf paths of this tree represent rows in $\mm{A}_k$. For any path $r, t_1, ..., t_i$ from $r$ to a node $t_i$, let $\overline{\mm{v}}_{t_i} = \mm{v}_{r} + \sum_{j=1}^{i-1} a_{t_j}\mm{v}_{t_j}$ where $a_{t_j}$ is $1$ if $t_{j+1}$ is the left child of $t_{j}$, and $-1$ otherwise. We will show that there exists a root-to-leaf $t_k$ path $t_1, ..., t_k$ such that $\left\|\overline{\mm{v}}_{t_k}\right\|_2^2 \ge k$. In particular we show that at every internal node $t$, with children $t_+$ and $t_-$, must have $\norm{\overline{\mm{v}}_{t_+}}^2 \geq 1 + \norm{\overline{\mm{v}}_t}^2$ or $\norm{\overline{\mm{v}}_{t_-}}^2 \geq 1 + \norm{\overline{\mm{v}}_t}^2$. To see this, note that
\[\norm{\overline{\mm{v}}_{t_+}}^2 = \norm{\overline{\mm{v}}_{t} + \mm{v}_{t}}^2 = \norm{\overline{\mm{v}}_{t}}^2 + \norm{\mm{v}_{t}}^2 + 2\anglebrac{\overline{\mm{v}}_{t}, \mm{v}_{t}} = \norm{\overline{\mm{v}}_{t}}^2 + 1 + 2\anglebrac{\overline{\mm{v}}_{t}, \mm{v}_{t}}.\]
Similarly, we have that $\norm{\overline{\mm{v}}_{t_-}}^2 = \norm{\overline{\mm{v}}_{t}}^2 + 1 - 2\anglebrac{\overline{\mm{v}}_{t}, \mm{v}_{t}}$. The claim then follows since either $\anglebrac{\overline{\mm{v}}_{t}, \mm{v}_{t}} \geq 0$ or $-\anglebrac{\overline{\mm{v}}_{t}, \mm{v}_{t}} \geq 0$. The theorem then follows from the tree interpretation of $\mm{A}_k$.
\end{proof}

In~\cite{Lovasz1986discrepancy} there appears an open problem of
S\'{o}s which asks if the hereditary discrepancy of a union of two
sets systems is bounded above by the discrepancy of each individual
set system i.e.\ for set systems $(X, \mathcal{S}_1)$ and $(X,
\mathcal{S}_2)$ is it true that $\herdisc(\mathcal{S}_1\cup
\mathcal{S}_2) \leq f(\herdisc(\mathcal{S}_1),
\herdisc(\mathcal{S}_2))$ for some function $f$? It turns out that no
such bound exists. The Hoffman example\footnote{Hoffman's set system
  $\mathcal{F}$ is defined on a regular $k$-ary tree of depth $k$ and
  obtains
  $\frac{\herdisc(\mm{A}_{\mathcal{F}})}{\detlb(\mm{A_{\mathcal{F}}})}
  = \Theta\left(\frac{\log n}{\log\log n}\right)$. Let $T$ be a
  $k$-regular tree with height $k$. The universe consists of the nodes
  of $T$. Let $\mathcal{F}_1$ be the sets of all root-to-leaf paths in
  $T$, and $\mathcal{F}_2$ be the set of all sibling sets (all nodes
  with the same parent) of internal nodes in $T$. Then $\mathcal{F} =
  \mathcal{F}_1 \cup \mathcal{F}_2$. We have that
  $\herdisc(\mathcal{F}_1), \herdisc(\mathcal{F}_2) \leq 1$,
  $\detlb(\mathcal{F}) = O(1)$, and $\disc(\mathcal{F}) =
  \Omega(k/\log k)$. See~\cite{matousek2009geometric} Section 4.4.},
the example of P{\'a}lv{\"o}lgyi~\cite{palvolgyi2010indecomposable},
and the three permutations family of Newman, Neiman, and Nikolov~\cite{newman2012beck} are instances of such pairs of set systems whose individual hereditary discrepancies are at most $1$, but whose union on a universe of size $n$ has discrepancy $\Omega(\log n/\log \log n)$ (for the Hoffman example) or $\Omega(\log n)$ (for the other two).

We see that $\mm{A}_k$ --- with the decomposition into $\mm{A}_k^+$
and $\mm{A}_k^-$ --- is another similar counter-example of S\'{o}s'
conjecture. While it matches the $\Omega(\log n)$ discrepancy lower
bound of the P{\'a}lv{\"o}lgyi and Newman-Neiman-Nikolov constructions, it is simpler to analyze.
\begin{claim}\label{claim:haar-is-Sos}
    With $\mm{A}_k^\pm$ as described above Claim~\ref{cor:haar-decomp}, 
    \[\disc\left(\mm{A}_k^{\pm}\right) \gtrsim k.\] 
\end{claim}
\begin{proof}
    Recall that $\disc(\mm{A}_k) = k$. We claim that $\disc(\mm{A}^\pm_k) \ge \frac12 \disc(\mm{A}_k)$, and this proves the claim. Indeed, take any coloring $\mm{x}$. Let $\mm{a}^{\top}$ be the row of $\mm{A}_k$ achieving $|\mm{a}^{\top}\mm{x}| = \disc(\mm{A}_k, \mm{x})$ and let $\mm{a}_+^{\top}$ and $\mm{a}_-^{\top}$ be the corresponding rows in the copy of $\mm{A}^{+}$ and $\mm{A}^-$ in $\mm{A}^{\pm}$ respectively. Since $\disc(\mm{A}_k) \leq |\mm{a}^{\top}\mm{x}| = |\mm{a}_+^{\top}\mm{x} - \mm{a}_-^{\top}\mm{x}|$, by the triangle inequality we have that either $|\mm{a}_+^{\top}\mm{x}| \geq \disc(\mm{A}_k)/2$ or $|\mm{a}_-^{\top}\mm{x}| \geq \disc(\mm{A}_k)/2$.
\end{proof}


\section{Upper Bound on Hereditary Discrepancy}
In this section we prove Theorem~\ref{thm:herdisc-ub}. To this end we introduce the volume lower bound on hereditary discrepancy, introduced by Lov\'{a}sz and Vesztergombi~\cite{LV89}, and, in a more general setting, by Banaszczyk~\cite{Bana-vollb}, and studied by Dadush, Nikolov, Talwar, and Tomczak-Jaegermann~\cite{dadush2018balancing}.

Let \(\mm{A}\) be an \(m\times n\) real matrix, and define the symmetric convex set $K_{\mm{A}} \coloneqq \{x \in \RR^n: \norm{\mm{Ax}}_{\infty}\leq 1\}$. Let us define the \emph{volume lower bound of $A$}, denoted $\mathrm{volLB}(A)$, by
\[\mathrm{volLB}(A) = \max_{k \in [n]}\max_{S \subseteq [n], |S| = k}\frac{1}{\mathrm{vol}_k\left(K_{\mm{A}} \cap W_S\right)^{1/k}},\]
where $W_S$ is the canonical subspace in the dimensions indexed by $S$ (i.e.\ $W_S = \mathrm{span}\left\{\mm{e}_i\mbox{, }i \in S\right\}$) and $\mathrm{vol}_k$ is the $k$-dimensional volume within $W_S$, i.e., the Lebesgue measure restricted to this subspace. We also define a dual volume lower bound by 
\[\mathrm{volLB}^*(A) = \max_{k \in [n]}\max_{S \subseteq [n], |S| = k}\frac{\mathrm{vol}_k\left(\mathrm{conv}(\pm \mm{\Pi}_S \mm{a}_1, \ldots \pm \mm{\Pi}_S \mm{a}_m)\right)^{1/k}}{c_k^{2/k}},\]
where $\mm{\Pi}_S$ is the orthogonal projection onto \(W_S\), $\mm{a_i}^\top$ is the \(i\)-th row of \(\mm{A}\), and \(c_k = \frac{\pi^{k/2}}{\Gamma(\frac{k}{2} + 1)}\) is the volume of the \(k\)-dimensional unit Euclidean ball. 

We also need the concept of a polar set of a set \(K \subseteq \RR^n\), defined as 
\[
K^\circ \coloneqq \{\mm{y} \in \RR^n: \mm{y}^\top \mm{x} \le 1 \ \forall \mm{x} \in K\}.
\]
It is a consequence of the hyperplane separator theorem that for any closed convex $K$ containing $0$, $K^{\circ\circ} = K$~\cite[Section 14]{Rockafellar}. Moreover, it is clear from the definition that $K\subseteq L$ implies $L^\circ \subseteq K^\circ$.

We have the following relationship between $\mathrm{volLB}(\mm{A})$ and $\mathrm{volLB}^*(\mm{A})$. 
\begin{claim}\label{claim:vollb-vs-vollb*}
    For any matrix $\mm{A} \in \RR^{m\times n}$, $\mathrm{volLB}(A) \cong \mathrm{volLB}^*(A)$.
\end{claim}
\begin{proof}
    Let $K_{\mm{A}}$ be defined as above, and let $L_{\mm{A}} \coloneqq \mathrm{conv}(\pm \mm{a}_1, \ldots \pm \mm{a}_m)$. We claim that, for any set \(S \subseteq [n]\), \((K_{\mm{A}} \cap W_S)^\circ = \mm{\Pi}_S L_{\mm{A}}\), where the polar \((K_{\mm{A}} \cap W_S)^\circ\) is taken within the subspace \(W_S\). It is sufficient to show this for $S = [n]$, as, otherwise, we can replace the matrix $\mm{A}$ by its submatrix consisting of the columns indexed by $S$. In the case $S = [n]$, we just need to show $K_{\mm{A}}^{\circ} = L_{\mm{A}}$. Notice that
    \begin{align*}
        K_{\mm{A}} &= \{\mm{x} \in \RR^{n}: \norm{\mm{A}\mm{x}}_{\infty} \leq 1\}\\
        &= \{\mm{x} \in \RR^{n}: \anglebrac{\mm{A}\mm{x}, \mm{y}} \leq 1 \mbox{ for all } \norm{\mm{y}}_1 \leq 1\}\\
        &= \{\mm{x} \in \RR^{n}: \anglebrac{\mm{x}, \mm{A}^{\top}\mm{y}} \leq 1 \mbox{ for all } \norm{\mm{y}}_1 \leq 1\}.
    \end{align*}
    By the definition of polar, we see that $K_{\mm{A}} = L_{\mm{A}}^{\circ}$ as 
    \[L_{\mm{A}} = \{\mm{A}^{\top}\mm{y}: \mm{y} \in \mathbb{R}^m\mbox{ where } \norm{\mm{y}}_1 \leq 1\}.\] 
    Thus $K_{\mm{A}}^{\circ} = L_{\mm{A}}^{\circ\circ} = L_{\mm{A}}$ as required. 
    
    Once we have established that \((K_{\mm{A}} \cap W_S)^\circ = \mm{\Pi}_S L_{\mm{A}}\), we have, by the Santal\'o-Blaschke and the reverse Santal\'o inequalities (see Chapters~1~and~8 of~\cite{ASGM15}),
    \[
    \mathrm{vol}_k(K_{\mm{A}} \cap W_S)^{1/k}\mathrm{vol}_k((K_{\mm{A}} \cap W_S)^\circ)^{1/k} \cong c_k^{2/k}.
    \]
    This completes the proof.
\end{proof}

The next lemma shows a relationship between $\mathrm{volLB}(\mm{A})$ and \(\detlb(\mm{A})\) that, as far as we are aware, has not been observed before. 
\begin{lemma}\label{lem:vollb-vs-detlb}
    For any matrix \(\mm{A} \in \RR^{m \times n}\), \(\mathrm{volLB}^*(\mm{A}) \lesssim \sqrt{n}\cdot\detlb(\mm{A}). \)
\end{lemma}

In the proof of Lemma~\ref{lem:vollb-vs-detlb} we use the following result of Nikolov. Closely related results were shown earlier by Dvoretzky and Rogers~\cite[Theorem 5B]{DR50} and Ball~\cite[Proposition 7]{Ball89}.

\begin{lemma}\textup{(Theorem 10 in~\cite{nikolov2015randomized}).}\label{lem:cite-subdet}
    Let $m \geq n$ and \(E\subseteq \RR^n\) be a minimum volume ellipsoid containing the points \(\pm \mm{a}_1, \ldots, \pm \mm{a}_m \in \RR^n\). There exists a set \(T\subseteq [m]\) of size $n$ such that 
    \[
    |\det((\mm{a}_i)_{i \in T})|
    \ge 
    \sqrt{\frac{n!}{n^n}} \frac{\mathrm{vol}_n(E)}{c_n} 
    \cong n^{1/4} e^{-n/2} \frac{\mathrm{vol}_n(E)}{c_n},
    \]
    where $(\mm{a}_i)_{i \in T}$ is the matrix with columns $\mm{a}_i$ for \(i \in T\), and \(\mathrm{vol}_n\) is the \(n\)-dimensional Lebesgue measure.
\end{lemma}
Note that Theorem 10 in \cite{nikolov2015randomized} in fact shows that there is a distribution on random multisets $T$ for which $\EE|\det((\mm{a}_i)_{i \in T})|^2 = \frac{n!}{n^n} \frac{\mathrm{vol}_n(E)^2}{c_n^2}$. Since the determinant is zero unless $T$ is a set, this implies Lemma~\ref{lem:cite-subdet}.

\begin{proof}[Proof of Lemma~\ref{lem:vollb-vs-detlb}]
    Take some $S\subseteq [n]$ of size $k$ such that 
    \[\mathrm{volLB}^*(A) = \frac{\mathrm{vol}_k(\mm{\Pi}_SL_{\mm{A}})^{1/k}}{c_k^{2/k}},\]
where $\mm{a}_1^\top, \ldots, \mm{a}_m^\top$ are the rows of $\mm{A}$, and $L_{\mm{A}} \coloneqq \mathrm{conv}(\pm \mm{a}_1, \ldots \pm \mm{a}_m)$. Applying Lemma~\ref{lem:cite-subdet} to $\pm \mm{\Pi}_S \mm{a}_1, \ldots, \pm \mm{\Pi}_S \mm{a}_m$, we have that, taking $E \subseteq W_S$ to be the smallest volume ellipsoid containing $\pm \mm{\Pi}_S \mm{a}_1, \ldots, \pm \mm{\Pi}_S \mm{a}_m$,  there exists a set $T \subseteq [m]$ of size $k$ for which 
\[
|\det((\mm{\Pi}_S \mm{a}_i)_{i \in T})|
    \gtrsim
     k^{1/4} e^{-k/2} \frac{\mathrm{vol}_k(E)}{c_k}
    \ge k^{1/4} e^{-k/2} \frac{\mathrm{vol}_k(\mm{\Pi}_S L_{\mm{A}})}{c_k}.
\]
The last inequality follows because $L_{\mm{A}} \subseteq E$. Re-arranging and raising to the power $1/k$, this gives us that 
\[
\mathrm{volLB}^*(A) = \frac{\mathrm{vol}_k(\mm{\Pi}_SL_{\mm{A}})^{1/k}}{c_k^{2/k}}
\lesssim
\frac{|\det((\mm{\Pi}_S \mm{a}_i)_{i \in T})|^{1/k}}{c_k^{1/k}}
\lesssim
\sqrt{k}\cdot\detlb(A),
\]
where, in the final inequality, we used that \((\mm{\Pi}_S \mm{a}_i)_{i \in T}\) is the transpose of a $k$ by $k$ submatrix of $\mm{A}$, and we also used the estimate $c_k^{-1/k}\lesssim \sqrt{k}$, which follows from Stirling's approximation. Since $k\le n$, the result follows.
\end{proof}

We remark in passing that the trivial inequality $\mathrm{vol}_k(E) \ge \mathrm{vol}_k(\mm{\Pi}_S L_{\mm{A}})$ for a $k$-dimensional symmetric convex polytope with $2m$ vertices $\mm{\Pi}_S L_{\mm{A}}$ and an ellipsoid $E$ containing it can be improved to $\mathrm{vol}_k(E) \ge \sqrt{\frac{k}{\log(2m)}}\mathrm{vol}_k(\mm{\Pi}_S L_{\mm{A}})$ when $m$ is small, using, e.g., results of Gluskin~\cite{gluskin1989extremal}. Substituting this inequality in the proof of Lemma~\ref{lem:vollb-vs-detlb} gives the bound 
\(
\mathrm{volLB}^*(\mm{A}) \lesssim \sqrt{\log 2m} \cdot \detlb(\mm{A}).
\)

The final ingredient we need for the proof of Theorem~\ref{thm:herdisc-ub} is an upper bound on the hereditary discrepancy of partial colorings in terms of the volume lower bound, due to Dadush, Nikolov, Tomczak-Jaegermann, and Talwar.
\begin{lemma}\textup{(Lem.8 in~\cite{dadush2018balancing}).}\label{lem:cite-balancing}
There exist universal constants $c \geq 1$ and $\epsilon_0 \in (0, 1)$ such that the following holds. For any closed convex set  $K \subseteq \RR^n$ satisfying \(-K = K\) and 
\[
\min_{k=1}^n\min_{S \subseteq [n]: |S| = k} \mathrm{vol}_k(K \cap W_S) \ge 1,
\]
and any $\mm{y} \in (-1, 1)^n$, there exists an $\mm{x} \in [-1, 1]^n$ with $|\mathrm{fixed}(\mm{x})| \geq \ceil{\epsilon_0n}$ and $\mm{x} - \mm{y} \in cK$, where \(\mathrm{fixed}(\mm{x}) \coloneqq \{i \in [n]: |x_i| = 1\}\).
\end{lemma}

We are now ready to complete the proof. 
\begin{proof}[Proof of Theorem~\ref{thm:herdisc-ub}]
It suffices to show that \(\disc(\mm{A}) \lesssim \sqrt{n}\ \detlb(\mm{A})\), since this implies that for any submatrix \(\mm{B}\) of \(\mm{A}\) with \(k\) columns we also have \[\disc(\mm{B}) \lesssim \sqrt{k}\ \detlb(\mm{B})\le \sqrt{n} \ \detlb(\mm{A}).\]

Using Lemma~\ref{lem:cite-balancing}, we construct a sequence of partial colorings $\mm{x}_0 = \mathbf{0}, \ldots, \mm{x}_T \in \{-1,+1\}^n$, where \(T \lesssim 1+ \log_{1/(1-\epsilon_0)}(n)\), each \(\mm{x}_t \in [0,1]^n\), and 
\begin{equation}\label{eq:partial-ub}
    \|\mm{A}(\mm{x}_t - \mm{x}_{t-1})\|_\infty \lesssim \sqrt{n(1-\epsilon_0)^{t-1}}\ \detlb(\mm{A}).
\end{equation}
To construct \(\mm{x}_1\), we apply Lemma~\ref{lem:cite-balancing} to \(\mm{y} \coloneqq \mathbf{0}\), and \(K \coloneqq \mathrm{volLB}(A)\cdot K_{\mm{A}}\). By the definition of \(\mathrm{volLB}(A)\), this \(K\) satisfies the assumption of the lemma, and we let \(\mm{x}_1\) equal the \(\mm{x}\) guaranteed by the lemma. Since \(\mm{x}_1 \in c K = c\cdot \mathrm{volLB}(A)\cdot K_{\mm{A}}\), by the definition of \(K_{\mm{A}}\) we have that 
\[
\|\mm{A}\mm{x}_1\|_\infty \le c\cdot \mathrm{volLB}(A)
\cong \mathrm{volLB}^*(A) \lesssim \sqrt{n}\ \detlb(\mm{A}),
\]
where the last two inequalities follow, respectively, by Claim~\ref{claim:vollb-vs-vollb*} and by Lemma~\ref{lem:vollb-vs-detlb}. In general, to get the bound \eqref{eq:partial-ub} for \(\mm{x}_t - \mm{x}_{t-1}\) for \(t \ge 2\), we set \(S \coloneqq [n] \setminus \mathrm{fixed}(\mm{x}_{t-1})\), and apply Lemma~\ref{lem:cite-balancing} with \(\mm{y} \coloneqq \mm{\Pi}_S\mm{x}_1\), and \(K \coloneqq \mathrm{volLB}(\mm{A}_S)\cdot K_{\mm{A}_S}\), where \(\mm{A}_S\) is the submatrix of \(\mm{A}\) consisting of the columns indexed by \(S\). If \(\mm{x} \in [-1,+1]^S\) is the partial coloring guaranteed by the lemma, we define \(\mm{x}_t\) by setting its coordinates in \(S\) to equal the corresponding coordinates in \(\mm{x}\), and the remaining coordinates to equal the corresponding coordinates in \(\mm{x}_{t-1}\). It is straightforward to check that \(\mathrm{fixed}(\mm{x}_t) \ge (1-(1-\epsilon_0)^t)n\) and \eqref{eq:partial-ub} hold for all \(t\). Moreover, once \(t \ge T\ge 1 + \log_{1/(1-\epsilon_0)}(n)\), we must have \(\mm{x}_t \in \{-1,+1\}^n\).

Having constructed \(\mm{x}_1, \ldots, \mm{x}_T\), we observe that, by \eqref{eq:partial-ub} and the triangle inequality,
\[
\|\mm{A}\mm{x}\|_\infty 
\lesssim \sqrt{n}\cdot\detlb(\mm{B})\left(1 + \sqrt{(1-\epsilon_0)} + \sqrt{(1-\epsilon_0)^2} + \cdots \right)\cong \sqrt{n}\cdot \detlb(\mm{A}).\]
This completes the proof.
\end{proof}

\color{black}
\section*{Acknowledgements}
The authors would like to thank the anonymous reviewers for their many useful comments. In particular, suggestions for claim~\ref{claim:vollb-vs-vollb*} helped clean up the proof. 


\bibliographystyle{alpha}
\bibliography{bibliography.bib}

\begin{thebibliography}{DNTTJ18}

\bibitem[AAGM15]{ASGM15}
Shiri Artstein-Avidan, Apostolos Giannopoulos, and Vitali~D. Milman.
\newblock {\em Asymptotic geometric analysis. {P}art {I}}, volume 202 of {\em
  Mathematical Surveys and Monographs}.
\newblock American Mathematical Society, Providence, RI, 2015.

\bibitem[Bal89]{Ball89}
Keith Ball.
\newblock Volumes of sections of cubes and related problems.
\newblock In {\em Geometric aspects of functional analysis (1987--88)}, volume
  1376 of {\em Lecture Notes in Math.}, pages 251--260. Springer, Berlin, 1989.

\bibitem[Ban93]{Bana-vollb}
Wojciech Banaszczyk.
\newblock Balancing vectors and convex bodies.
\newblock {\em Studia Math.}, 106(1):93--100, 1993.

\bibitem[Ban98]{banaszczyk1998balancing}
Wojciech Banaszczyk.
\newblock Balancing vectors and gaussian measures of n-dimensional convex
  bodies.
\newblock {\em Random Structures \& Algorithms}, 12(4):351--360, 1998.

\bibitem[Ban10]{bansal2010constructive}
Nikhil Bansal.
\newblock Constructive algorithms for discrepancy minimization.
\newblock In {\em Foundations of Computer Science (FOCS), 2010 51st Annual IEEE
  Symposium on}, pages 3--10. IEEE, 2010.

\bibitem[BDG16]{bansal2016algorithm}
Nikhil Bansal, Daniel Dadush, and Shashwat Garg.
\newblock An algorithm for komlos conjecture matching banaszczyk's bound.
\newblock In {\em Foundations of Computer Science (FOCS), 2016 IEEE 57th Annual
  Symposium on}, pages 788--799. IEEE, 2016.

\bibitem[CL01]{chazellelvov2001discrepancy}
Bernard Chazelle and Alexey Lvov.
\newblock The discrepancy of boxes in higher dimension.
\newblock {\em Discrete \& Computational Geometry}, 25:519--524, 2001.

\bibitem[DNTTJ18]{dadush2018balancing}
Daniel Dadush, Aleksandar Nikolov, Kunal Talwar, and Nicole Tomczak-Jaegermann.
\newblock Balancing vectors in any norm.
\newblock In {\em 2018 IEEE 59th Annual Symposium on Foundations of Computer
  Science (FOCS)}, pages 1--10. IEEE, 2018.

\bibitem[DR50]{DR50}
A.~Dvoretzky and C.~A. Rogers.
\newblock Absolute and unconditional convergence in normed linear spaces.
\newblock {\em Proc. Nat. Acad. Sci. U.S.A.}, 36:192--197, 1950.

\bibitem[Fra18]{franks2018simplified}
Cole Franks.
\newblock A simplified disproof of beck's three permutations conjecture and an
  application to root-mean-squared discrepancy.
\newblock {\em arXiv preprint arXiv:1811.01102}, 2018.

\bibitem[GH62]{ghouila1962caracterisation}
Alain Ghouila-Houri.
\newblock Caract{\'e}risation des matrices totalement unimodulaires.
\newblock {\em Comptes Redus Hebdomadaires des S{\'e}ances de l'Acad{\'e}mie
  des Sciences (Paris)}, 254:1192--1194, 1962.

\bibitem[Glu89]{gluskin1989extremal}
Efim~Davydovich Gluskin.
\newblock Extremal properties of orthogonal parallelepipeds and their
  applications to the geometry of banach spaces.
\newblock {\em Mathematics of the USSR-Sbornik}, 64(1):85, 1989.

\bibitem[JR22]{jiang2022tighter}
Haotian Jiang and Victor Reis.
\newblock A tighter relation between hereditary discrepancy and determinant
  lower bound.
\newblock In {\em Symposium on Simplicity in Algorithms (SOSA)}, pages
  308--313. SIAM, 2022.

\bibitem[Kun23]{kunisky2023discrepancy}
Dmitriy Kunisky.
\newblock The discrepancy of unsatisfiable matrices and a lower bound for the
  koml{\'o}s conjecture constant.
\newblock {\em SIAM Journal on Discrete Mathematics}, 37(2):586--603, 2023.

\bibitem[Lav18]{lavrov2018}
Misha Lavrov.
\newblock Finding sum of k from zero to n of the absolute value of twice k
  minus n times n choose k, 2018.

\bibitem[LSV86]{Lovasz1986discrepancy}
L{\'a}szl{\'o} Lov{\'a}sz, Joel Spencer, and Katalin Vesztergombi.
\newblock Discrepancy of set-systems and matrices.
\newblock {\em European Journal of Combinatorics}, 7(2):151--160, 1986.

\bibitem[LV89]{LV89}
L.~Lov\'{a}sz and K.~Vesztergombi.
\newblock Extremal problems for discrepancy.
\newblock In {\em Irregularities of partitions ({F}ert\H{o}d, 1986)}, volume~8
  of {\em Algorithms Combin. Study Res. Texts}, pages 107--113. Springer,
  Berlin, 1989.

\bibitem[Mat95]{matousek95tight}
J.~Matou{\v{s}}ek.
\newblock {Tight Upper Bounds for the Discrepancy of Halfspaces}.
\newblock {\em Discrete and Computational Geometry}, 13(1):593--601, 1995.

\bibitem[Mat09]{matousek2009geometric}
Ji{\v{r}}{\'{i}} Matou{\v{s}}ek.
\newblock {\em Geometric discrepancy: An illustrated guide}, volume~18.
\newblock Springer Science \& Business Media, 2009.

\bibitem[Mat13]{matouvsek2013determinant}
Ji{\v{r}}{\'\i} Matou{\v{s}}ek.
\newblock The determinant bound for discrepancy is almost tight.
\newblock {\em Proceedings of the American Mathematical Society},
  141(2):451--460, 2013.

\bibitem[MNT18]{matousek18factorization}
Jiří Matoušek, Aleksandar Nikolov, and Kunal Talwar.
\newblock {Factorization Norms and Hereditary Discrepancy}.
\newblock {\em International Mathematics Research Notices}, 2020(3):751--780,
  03 2018.

\bibitem[Nik15]{nikolov2015randomized}
Aleksandar Nikolov.
\newblock Randomized rounding for the largest simplex problem.
\newblock In {\em Proceedings of the forty-seventh annual ACM symposium on
  Theory of computing}, pages 861--870, 2015.

\bibitem[Nik17]{nikolov17tusnady}
A.~Nikolov.
\newblock Tighter bounds for the discrepancy of boxes and polytopes.
\newblock {\em Mathematika}, 63(3):1091--1113, 2017.

\bibitem[NNN12]{newman2012beck}
Alantha Newman, Ofer Neiman, and Aleksandar Nikolov.
\newblock Beck's three permutations conjecture: A counterexample and some
  consequences.
\newblock In {\em 2012 IEEE 53rd Annual Symposium on Foundations of Computer
  Science}, pages 253--262. IEEE, 2012.

\bibitem[P{\'a}l10]{palvolgyi2010indecomposable}
D{\"o}m{\"o}t{\"o}r P{\'a}lv{\"o}lgyi.
\newblock Indecomposable coverings with concave polygons.
\newblock {\em Discrete \& Computational Geometry}, 44(3):577--588, 2010.

\bibitem[Roc70]{Rockafellar}
R.~Tyrrell Rockafellar.
\newblock {\em Convex analysis}, volume No. 28 of {\em Princeton Mathematical
  Series}.
\newblock Princeton University Press, Princeton, NJ, 1970.

\bibitem[Sch98]{schrijver1998theory}
Alexander Schrijver.
\newblock {\em Theory of linear and integer programming}.
\newblock John Wiley \& Sons, 1998.

\bibitem[Ver18]{vershynin2018high}
Roman Vershynin.
\newblock {\em High-dimensional probability: An introduction with applications
  in data science}, volume~47.
\newblock Cambridge university press, 2018.

\end{thebibliography}
\appendix
\section{Upper Bound on Hereditary Discrepancy for Set Systems}

\modified{Here we give a different proof of the special case of Theorem~\ref{thm:herdisc-ub} for matrices $\mm{A}$ with entries in $\{0,1\}$, i.e., incidence matrices of set systems. This proof is more elementary, and has the interesting property that the  upper bound on discrepancy is certified by a uniformly random coloring.}

We need is a connection between the hereditary discrepancy of a set system $\mathcal{S}$ on a universe $X$, and its VC dimension. Recall that the VC dimension, denoted $\dim(\mathcal{S})$, is the largest size of a set $Y\subseteq X$ such that the restriction $\mathcal{S}|_Y \coloneqq \{S \cap Y: S \in \mathcal{S}\}$ equals the powerset of $Y$. Equivalently, we can define the VC dimension $\dim(\mm{A})$ of an incidence matrix $\mm{A} \in \{0,1\}^{m\times n}$ as the largest $N$ for which the power matrix $\mm{P}_N$ can be found as a submatrix of $\mm{A}$. 

\subsection{Upper Bound in Terms of VC Dimension}

Let us introduce some terminology from stochastic processes that we use in our proof. For a metric space $(T, d)$, let $\mathcal{N}(T, d, \epsilon)$ be the \emph{covering number of $T$} i.e.\ $\mathcal{N}(T, d, \epsilon)$ is the smallest number of closed balls with centers in $T$ and radii $\epsilon$ whose union covers $T$. Further, let $\norm{Y}_{\psi_2}$ be the sub-gaussian norm of a real-valued random variable $Y$ where $\norm{Y}_{\psi_2} \coloneqq \inf\{t \geq 0: \EE \exp\left(Y^2/t^2\right) \leq 2\}$. Finally, for a class of real-valued functions $\mathcal{F}$ defined on a probability space $(\Omega, \mu)$, where \(\Omega\) is a finite set, we define the $L^2(\mu)$ norm by
\(
\|f\|_{L^2(\mu)} := \left(\sum_{\omega \in \Omega} |f(\omega)|^2\mu(\omega)\right)^{1/2}.
\)
If $\mu$ is the uniform measure on $\Omega$, then we simple write $L^2$ rather than $L^2(\mu)$. 

The following result is a consequence of well-known lemmas. We recount it here for completeness.

\begin{lemma}\label{lem:herdisc-vc-ub}
    \modified{For any non-constant matrix $\mm{A}\in \{0,1\}^{m\times n}$, 
    \[\herdisc(\mm{A}) \lesssim\sqrt{n\cdot\dim(\mm{A})}.\]}
\end{lemma}
\begin{proof}
    \modified{
    It is enough to prove that $\disc(\mm{A}) \lesssim \sqrt{n\cdot\dim(\mm{A})}$ since, applying this inequality to any submatrix $\mm{B}$ consisting of a subset of $k$ columns from $\mm{A}$ shows that, $\disc(\mm{B}) \lesssim \sqrt{k\cdot\dim(\mm{B})} \le \sqrt{n\cdot\dim(\mm{A})}$.}
    
    \modified{We prove that $\disc(\mm{A}, \mm{x}) \lesssim \sqrt{n\cdot\dim(\mm{A})}$ is satisfied in expectation by a uniformly random coloring $\mm{x} \in \{-1, +1\}^{n}$ with entries $X_1, X_2, ... X_{n}$ which are independent Rademacher random variables (i.e., uniform in \(\{-1,+1\}\)). Let $\mathcal{F}$ denote the class of indicator functions defined by the rows of $\mm{A}$, i.e., for every $i \in [m]$, we define a function $f_{i}: \left[n\right] \rightarrow \{0,1\}$ given by $f_i(j) = a_{i,j}$.} We will show that 
    \[\EE\sup_{f\in \mathcal{F}} \left|\frac{1}{\sqrt{n}}\sum_{i = 1}^{n}X_if\left(i\right)\right| \lesssim \sqrt{\dim(\mathcal{F})}.\]
    
    For each indicator function $f$, let the random variable $Z_f \coloneqq \left|\frac{1}{\sqrt{n}}\sum_{i=1}^{n} X_if(i)\right|$. Note that this differs from the discrepancy of the row indicated by $f$ by a multiple of $\sqrt{n}$, i.e., $\disc(\mm{A},\mm{x}) = \sqrt{n} \cdot \sup_{f \in \mathcal{F}} Z_f$. Consider the random process $(Z_f)_{f\in \mathcal{F}}$. We will apply Dudley's inequality (Lemma~\ref{lem:v-dudley}) to show that 
    \begin{equation}\label{eq:dudley}
        \EE\sup_{f \in \mathcal{F}} Z_f \lesssim \int_0^1\sqrt{\log \mathcal{N}\left(\mathcal{F}, L^2, \epsilon\right)}d\epsilon.
    \end{equation}
    In order to apply Dudley's inequality we must show that $(Z_f)_{f\in \mathcal{F}}$ has sub-gaussian increments. Note that, since $\norm{X_i}_{\psi_2} \lesssim 1$, we have 
    \begin{align*}
        \norm{Z_f - Z_g}_{\psi} &= \frac{1}{\sqrt{n}}\left\|\left|\sum_{i=1}^{n}X_if(i)\right|-\left|\sum_{i=1}^{n}X_ig(i)\right|\right\|_{\psi_2}\\ 
        &\leq \frac{1}{\sqrt{n}}\left\|\sum_{i=1}^{n}X_i(f-g)(i)\right\|_{\psi_2} \lesssim \left(\frac{1}{n}\sum_{i=1}^n(f-g)(i)^2\right)^{1/2},
    \end{align*}
    where the second step follows from the reverse triangle inequality, and the final step by Hoeffding's lemma. The right hand side is $\norm{f - g}_{L^2}$ so when we apply Dudley's inequality as shown in Lemma~\ref{lem:v-dudley}, we obtain Equation~\ref{eq:dudley}. 
    
    Using Theorem~\ref{thm:v-covering-number}, we can bound the covering number with respect to the normalized $L_2$ norm as 
    \[\log \mathcal{N}\left(\mathcal{F}, L^2, \epsilon\right) \lesssim \dim(\mathcal{F})\log \left(\frac{2}{\epsilon}\right).\]
    Plugging the right hand side into the integral in Equation~\ref{eq:dudley} and integrating, we have $\EE \sup_{f \in \mathcal{F}}Z_f \lesssim \sqrt{\dim(\mathcal{F})}$. Recall that the discrepancy of the row indicated by $f$ is $\sqrt{n}\cdot Z_f$, thus the hereditary discrepancy is bounded above as $\herdisc(\mm{A}) \lesssim\sqrt{n\dim(\mathcal{F})}$, as was our goal.
\end{proof}

\begin{lemma}{\textup{(Dudley's Inequality, Remark 8.1.5~\cite{vershynin2018high}).}}\label{lem:v-dudley}
    Let $(X_t)_{t\in T}$ be a random process on a metric space $(T, d)$ with sub-gaussian increments i.e.\ there exists a $K \geq 0$ such that $\norm{X_t - X_s}_{\psi_2} \leq Kd(t,s)$ for all $t, s \in T$. Then 
    \[\EE \sup_{t,s \in T}\left|X_t - X_s\right| \lesssim K\int_0^{\infty}\sqrt{\log \mathcal{N}(T, d, \epsilon)}d\epsilon.\]
\end{lemma}

\begin{theorem}\textup{(Covering Numbers via VC Dimension, 8.3.18~\cite{vershynin2018high}).}\label{thm:v-covering-number}
    Let $\mathcal{F}$ be a class of Boolean functions on a probability space $(\Omega, \Sigma, \mu)$. Then, for every $\epsilon \in (0, 1)$, we have
    \[\mathcal{N}\left(\mathcal{F}, L^2(\mu), e\right) \leq \left(\frac{2}{\epsilon}\right)^{C\cdot\dim(\mathcal{F})}\]
    for an absolute constant $C$.
\end{theorem}
We note that Lemma~\ref{lem:herdisc-vc-ub} can likely be improved
further, for example by following the techniques of
Matou\v{s}ek~\cite{matousek95tight}, and carefully tracking constants.

\subsection{Connection to the Determinant Lower Bound}

\modified{To finish the proof, it remains to show a connection between VC dimension and the determinant lower bound. To do so, we show that a matrix \(\mm{A}\in \{0,1\}^{m\times n}\) with large VC dimension must contain a submatrix with large determinant. This submatrix is a binary version of the Hadamard matrix, described next.}

\modified{Let the $0$-$1$ Hadamard matrix be the $\{0,1\}$ matrix obtained by applying the linear map $a \mapsto (a + 1)/2$ to all of the entries in the standard $\pm 1$ Hadamard matrix. Denote the $n\times n$ $0$-$1$ and standard Hadamard matrices by $\tilde{\mm{H}}_n$ and $\mm{H}_n$ respectively. We prove the following. 
\begin{claim}\label{claim:01-hadamard-determinant}
    $\left|\det\left(\tilde{H}_n\right)\right| \geq 2^{-n}\cdot n^{n/2}$.
\end{claim}
\begin{proof}
    Consider $\mm{H}_n$ and suppose w.l.o.g.\ that its first rows is the all ones row. Add this row to all the other rows. Observe that all the other rows now have entries in $\{0, 2\}$. Scale them down by a factor of two. Adding one row to another does not change the determinant. Scaling a row scales the determinant by the same amount. Since $\left|\det\left(\mm{H}_n\right)\right| = n^{n/2}$, $\left|\det\left(\tilde{\mm{H}}_n\right)\right| = 2^{-n}\cdot n^{n/2}$.   
\end{proof}}

We can now finish the proof of Theorem~\ref{thm:herdisc-ub} for \(\mm{A} \in \{0,1\}^{m\times n}\). If $\mm{A}$ is a constant matrix (i.e., all its entries are equal), the bound is trivial, so we assume otherwise. The upper bound arises from the pair of inequalities $\herdisc(\mm{A}) \lesssim \sqrt{n\dim(\mm{A})}$ and
$\sqrt{\dim(\mm{A})} \lesssim \detlb(\mm{A})$. The former inequality is achieved by a random coloring, as shown in Lemma~\ref{lem:herdisc-vc-ub}. \modified{The latter follows by considering the power matrix $\mm{P}_{\dim(\mm{A})}$, which is a submatrix of $\mm{A}$. Since every $\dim(\mm{A}) \times \dim(\mm{A})$ 0-1 matrix is a submatrix of $\mm{P}_{\dim(\mm{A})}$, we can find also find $\tilde{\mm{H}}_{\dim(\mm{A})}$ as a submatrix of $\mm{P}_{\dim(\mm{A})}$, and, therefore, of $\mm{A}$. By Claim~\ref{claim:01-hadamard-determinant} we know that $\left|\det\left(\tilde{\mm{H}}_{\dim(\mm{A})}\right)\right| \geq 2^{-\dim(\mm{A})}\cdot\left(\dim(\mm{A})\right)^{\dim(\mm{A})/2}$. It follows that
    \[\detlb(\mm{A})  \geq \left|\det\left(\tilde{\mm{H}}_{\dim(\mm{A})}\right)\right|^{1/\dim(\mm{A})} \gtrsim \sqrt{\dim(\mm{A})}.\]
    Thus, the two inequalities together give us $\detlb(\mm{A}) \gtrsim \herdisc(\mm{A})/\sqrt{n}$ as required.}


\end{document}